\documentclass[smallextended]{svjour3}       % onecolumn (second format)
\smartqed  % flush right qed marks, e.g. at end of proof
\usepackage{graphicx}
\usepackage{amssymb,amsmath, amsfonts}
\usepackage[algo2e,linesnumbered,noend,noline]{algorithm2e}
\usepackage{algorithmic,refcount}
\usepackage{algorithm}
\usepackage[utf8]{inputenc}
\usepackage[dvipsnames]{xcolor}
\usepackage[shortlabels]{enumitem}
\newtheorem{Theorem}{Theorem}
\newtheorem{Lemma}{Lemma}
\newtheorem{Claim}{Claim}
\newtheorem{Proposition}{Proposition}

\newcommand{\X}{{\sf X} } 
\newcommand{\Y}{{\sf Y} } 
\newcommand{\Z}{{\sf Z} } 
\begin{document}

\title{The $(2,k)$-connectivity augmentation problem:\\
Algorithmic aspects}%\thanks{Grants or other notes


\author{Florian H\"orsch*    \and
        Zolt\'an Szigeti %etc.
}


\institute{Florian H\"orsch \at
              46 Avenue F\'elix Viallet\\
		Grenoble, France\\
              \email{florian.hoersch@grenoble-inp.fr}\\
		Corresponding author           %  \\
           \and
          Zolt\'an Szigeti \at
              46 Avenue F\'elix Viallet\\
		Grenoble, France\\
              \email{zoltan.szigeti@grenoble-inp.fr}
}


%\iffalse

\maketitle

\begin{abstract}
Durand de Gevigney and Szigeti \cite{DgGSz} have recently given a min-max theorem for the $(2,k)$-connectivity augmentation problem.
This article provides an $O(n^3(m+ n \textrm{ }\log\textrm{ }n))$ time algorithm to find an optimal solution for this problem.
\keywords{Connectivity augmentation}
\end{abstract}

\section{Introduction}

Let {\boldmath$G=(V,E)$ be an undirected graph with $n$ vertices and $m$ edges and let $c$} $:E\rightarrow \mathbb{Z}_{> 0}$ be an integer edge capacity function. For $X \subseteq V$, we denote the set of edges with exactly one endvertex in $X$ by {\boldmath $\delta_G(X)$}. We say that $G$ is {\it $k$-edge-connected} if $\sum_{e \in \delta_G(X)}c(e)\geq k$ for all nonempty, proper $ X\subset V.$ 
\smallskip

For all problems treated in this article, the version with capacities and the version without capacities can be easily reduced to each other by replacing an edge with a capacity by multiple edges and vice-versa. Yet, this does not mean algorithmic equivalence. All the running times we give hold for the case with edge capacities, assuming that all basic operations can be executed in constant time. For this reason, all technical statements starting from Section \ref{pre} will be given in the capacitated form. During the introduction, however, we first describe the problems in the uncapacitated form for the sake of simplicity.
\smallskip
 
The theory of graph connectivity augmentation has seen significant progress during the last decades. 
The basic problem, the {\it global edge-connectivity augmentation} of undirected graphs, can be defined as follows:
Given an undirected graph $G$ and a positive integer $k$, find a set of edges of minimum cardinality whose addition results in a $k$-edge-connected graph.
The problem was solved in terms of a min-max theorem by Cai and Sun \cite {CS} and a polynomial time algorithm was provided by Watanabe and Nakamura \cite{WN}. If we replace the minimum cardinality condition in the above problem by a minimum cost condition with respect to a given arbitrary cost function on $V^2$, the problem becomes $NP$-complete as proven by Eswaran and Tarjan in \cite{ET}. For this reason, we only consider the minimum cardinality version for all augmentation problems treated in this article.
\smallskip

The method that is nowadays most commonly used when dealing with connectivity augmentation problems was introduced by Frank \cite{F}. This method consists of two steps. In the first one a new vertex as well as edges connecting it to the given graph are added to have the required connectivity condition. Such a graph with the minimum even number of new edges is called a {\it minimal even extension}. The second step applies an operation called splitting off. A splitting off is the deletion of two edges incident to the new vertex and the addition of an edge between the other endvertices of these edges. A minimum augmentation of $G$ can be obtained by repeatedly applying splitting offs maintaining the connectivity requirements and finally deleting the added vertex.  
\smallskip

Besides the basic case, this method allows us to handle several other versions of the problem such as the local edge-connectivity augmentation problem in undirected graphs and the global arc-connectivity augmentation problem for directed graphs. Frank \cite{F} managed to provide both min-max theorems and efficient algorithms for these problems. Several further applications have been found, see for example \cite{BAB} -- \cite{BGS2} and \cite{S}.
\smallskip

Vertex-connectivity augmentation problems are more complicated than edge-connectivity augmentation problems. For the global vertex-connectivity augmentation problem in undirected graphs, no  min-max theorem is known, however, a polynomial time algorithm for  constant $k$ was given by Jackson and Jord\'an \cite{JJ}. For the global vertex-connectivity augmentation problem in directed graphs, a min-max theorem and an efficient algorithm for constant $k$ were given by Frank and Jord\'an \cite{FJ}. The main contribution of \cite{FJ}, besides the solution of the directed global vertex-connectivity augmentation problem, is the first application of bisets, a technique we also rely on in the present article.
\smallskip

Returning to the basic case of global undirected edge-connectivity augmentation, Frank \cite{F} provided an $O(n^5)$ time algorithm based on the approach described above. In order to achieve this running time, he uses a slightly involved method for the splitting off part. Carefully choosing the pairs of edges to be split off, he manages to finish with a complete splitting off after a linear number of splitting off operations in $n$ while the obvious approach results in a quadratic number. Nagamochi and Ibaraki \cite{NI} provided a more efficient implementation of the method described above. They managed to find a minimum $k$-edge-connected augmentation in $O((n \textrm{ }\log\textrm{ } n) (m+n \textrm{ }\log\textrm{ } n))$ time. An important ingredient in their work is an $O(n(m+ n \textrm{ }\log\textrm{ } n))$ time mincut algorithm. We also make use of this mincut algorithm as a subroutine.
\smallskip

The problem we deal with in this article is $(2,k)$-connectivity augmentation. The concept of $(2,k)$-connectivity is a mixed form of edge-connectivity and vertex-connectivity and was first introduced in a more general form by Kaneko and Ota \cite{KO}.
A graph $G=(V,E)$ is said to be $(2,k)$-connected if $|V|\ge 3$, $G$ is $2k$-edge-connected and $G-v$ is $k$-edge-connected for every vertex $v\in V.$ Bisets provide a convenient framework to treat this mixed-connectivity concept. %It turned out (see Lemma 3.4 in \cite{DgGSz}) that the two connectivity conditions do not interfere.
The $(2,k)$-connectivity augmentation problem was considered by Durand de Gevigney and Szigeti \cite{DgGSz}, where a min-max theorem was proved.  
The authors characterized graphs admitting a complete splitting off maintaining $(2,k)$-connectivity in terms of an obstacle. They also worked out the minimal even extension step for $(2,k)$-connectivity. They showed that a minimal even extension for $(2,k)$-connectivity exists so that the obtained graph contains no obstacle yielding a complete admissible splitting off. This allowed the application of Frank's method  to derive a min-max theorem.
While a careful analysis of their proof yields a rather slow polynomial time algorithm for the uncapacitated version, no explicit mention of an algorithm is made in \cite{DgGSz}. 
\smallskip

The aim of the present article is to give an explicit algorithm of significantly improved running time that works for the more general  capacitated  version. Given a graph $G=(V,E)$ and $k\geq 2$, a {\it minimum $(2,k)$-connected augmentation} of $G$ is a $(2,k)$-connected graph $G'=(V,E')$ such that $E \subseteq E'$ and $|E'|$ is minimum. The main result of this article is the following:

\begin{Theorem}
Given a graph $G$ and an integer $k \geq 2$, we can compute a minimum $(2,k)$-connected augmentation of $G$ in $O(n^3(m+ n \textrm{ }\log\textrm{ } n))$ time.
\end{Theorem}

\smallskip
During the whole article we suppose that $k \geq 2$. For $k=1$, observe that $(2,1)$-connectivity is equivalent to $2$-vertex-connectivity and an $O(n+m)$ time algorithm for $2$-vertex-connectivity augmentation has been given in \cite{ET} by Eswaran and Tarjan. We follow the proof technique of \cite{DgGSz} and show how to turn it  into an efficient algorithm with some more effort. An efficient algorithm for finding a minimal even extension for $(2,k)$-connectivity comes straight from its definition. Therefore, the main difficulty is to find the complete admissible splitting off yielding a $(2,k)$-connected augmentation of $G$ whose existence is guaranteed by \cite{DgGSz}. In comparison to the problems in \cite{F}, we seem to face another difficulty; graphs containing an obstacle need to be avoided during the splitting off process. We show that no obstacle can ever be created when splitting offs are executed from a minimal even extension for $(2,k)$-connectivity, so the aforementioned difficulty does not actually exist. We first use the mincut algorithm of \cite{NI} to compute the best splitting off for a given pair of  edges incident to the new vertex in $O(n^2(m+n \textrm{ }\log\textrm{ } n))$ time. This yields a rather simple algorithm for the $(2,k)$-connectivity augmentation problem that runs in $O(n^4(m+ n \textrm{ }\log\textrm{ }n))$ time. In the last part, we provide a more efficient algorithm for the splitting off part. In order to speed up the previous algorithm, we choose the splitting offs in a way that linearizes the number of required splitting offs. The overall running time of our algorithm for $(2,k)$-connectivity augmentation is $O(n^3(m+ n \textrm{ }\log\textrm{ }n))$.\\

This algorithm is inspired by the method used by Frank in \cite{F} for the case of $k$-edge-connectivity. However, there is a siginificant difference between our method and the method used by Frank. While Frank tries to execute several splitting offs all of which contain one common edge, we try to split off all pairs of edges included in a well-chosen triple of edges. This allows us to avoid a significant amount of technicalities in comparison to an earlier version of this article.  
\smallskip

This article is organized as follows: In Section \ref{pre}, we give an overview of the results we use and show how to turn them into subroutines for our algorithm. In Section \ref{ext}, we show how to efficiently compute a minimal even extension for $(2,k)$-connectivity and give some of its important properties. We show how to turn this into a simple algorithm that runs in $O(n^4(m+ n \textrm{ }\log\textrm{ }n))$ time. In Section \ref{algo}, we present a more involved algorithm that runs in $O(n^3(m+ n \textrm{ }\log\textrm{ }n))$ time.
%We present an efficient algorithm to find an optimal solution to the global edge-connectivity augmentation problem that uses $O(n^3)$ global min-cut, and hence, using Nagamochi and Ibaraki's algorithm \cite{NI}, its  running time is $O(n^4(m+n\log(n)))$. 

\section{Basic definitions and previous results}\label{pre}

\subsection{Capacitated graphs}

All the graphs considered are undirected and loopless. On the other hand, our graphs have nonnegative integer edge capacities. Let $(G=(V,E),c)$ be a capacitated graph, i.e. a graph with an integer edge capacity function $c:E\rightarrow \mathbb{Z}_{> 0}$. Given a set $F \subseteq E$ of edges, we will use {\boldmath $c(F)$} for $\sum_{e \in F}c(e)$. Observe that we require all our capacitated graphs to have integer edge capacities, so we shall not create any capacitated graphs with noninteger edge capacities during our algorithm. We also do not consider capacitated graphs with multiple edges as they can easily be replaced by a single edge with the sum of the corresponding capacities. For our algorithms, we will assume that additions and subtractions of real numbers can be executed in constant time.

 For  $X,Y\subseteq V$, we use {\boldmath $\delta_G(X,Y)$} to denote the set of edges between $X-Y$ and $Y-X$. We use {\boldmath $\overline{\delta}_G(X,Y)$} for $\delta_G(X\cap Y,V-(X \cup Y))$.
For a vertex $v \in V$, we use {\boldmath$\delta_G(v)$} for $\delta_G(\{v\})$ and {\boldmath$N_G(v)$} for the set of neighbors of $v$.% We say that $G$ is {\it $k$-edge-connected} if $\sum_{e \in \delta_G(X)}c(e)\geq k$ for all nonempty, proper $ X\subset V.$ 

We require the following well-known result that can be found in Proposition 1.2.1 in \cite{book}.

\begin{Proposition}\label{ftyig} Let $(G=(V,E),c)$ be a capacitated graph. For all $X,Y\subseteq V,$ the following hold:
	\begin{enumerate}[(a)]
		\item $c(\delta_G(X))+c(\delta_G(Y))=c(\delta_G(X\cap Y))+c(\delta_G(X\cup Y))+2c(\delta_G(X,Y))$,
		\item $c(\delta_G(X))+c(\delta_G(Y))=c(\delta_G(X-Y))+c(\delta_G(X-Y))+2c(\overline{\delta}_G(X,Y))$.
	\end{enumerate}
\end{Proposition}

\subsection{Minimum cuts}
%Given a graph $G=(V,E)$ and two disjoint sets $S_1,S_2 \subseteq V$ of vertices, the associated cut $\delta(S_1,S_2)$ is the set of edges that have one end in $S_1$ and one end in $S_2$. We say $\delta(S)$ for $\delta(S,V-S)$. Observe that $\delta(S)=\delta(V-S)$.
Given a capacitated graph $(G=(V,E),c)$, the mincut problem consists of finding a set $\emptyset \neq S \subsetneq V$ that minimizes $c(\delta_G(S))$. This problem has been widely studied. Due to the specific nature of our application, we are interested in a slight variation of this problem: given a capacitated graph $(H=(V+s,E),c)$ with a distinguished vertex $s \notin V$, we want to find a set $\emptyset \neq S \subsetneq V$ that minimizes $c(\delta_H(S))$.  In other words, we additionally require that no side of the cut consists of $s$ only. We denote the capacity of such a minimum cut by {\boldmath $\lambda_{(H,c)}(V)$}. 
We say that $(H=(V+s,E),c)$ is {\it $k$-edge-connected in $V$} for some positive integer $k$ if $\lambda_{(H,c)}(V) \geq k$. 
We strongly rely on the following algorithmic result which is due to Nagamochi and Ibaraki \cite{NI}.  

\begin{Lemma}\label{japonais}
Given a capacitated graph $(H=(V+s,E),c)$, we can compute $\lambda_{(H,c)}(V)$ and a set $\emptyset \neq S \subsetneq V$ that minimizes $c(\delta_H (S))$ in $O(n(m+n \textrm{ }\log\textrm{ } n))$ time.
\end{Lemma}

\subsection{Bisets}

Given a ground set $\Omega$, a {\it biset} {\boldmath${\sf X}$} consists of two sets {\boldmath$X_O,X_I$} $\subseteq \Omega$ with $X_I \subseteq X_O$. We call $X_I$ the {\it inner set} of ${\sf X}$, $X_O$ the {\it outer set} of ${\sf X}$ and {\boldmath$w({\sf X})$}$=X_O-X_I$ the {\it wall} of ${\sf X}$. We also say ${\sf X}=(X_O,X_I)$. We define the {\it complement} {\boldmath $\overline{{\sf X}}$} of $\X$ by $\overline{X}_O=\Omega -X_I$ and $\overline{X}_I=\Omega -X_O$. Observe that $w(\overline{{\sf X}})=w({\sf X})$.  We say that ${\sf X}$ is {\it trivial} with respect to $\Omega$ if $X_I = \emptyset$ or $X_O=\Omega$, {\it nontrivial} otherwise.
Given two bisets ${\sf X}$ and ${\sf Y}$, we define {\boldmath ${\sf X} \cup {\sf Y}$}$=(X_O \cup Y_O, X_I \cup Y_I)$, {\boldmath${\sf X} \cap {\sf Y}$}$=(X_O \cap Y_O, X_I \cap Y_I)$ and {\boldmath ${\sf X}-{\sf Y}$}$={\sf X} \cap \overline{{\sf Y}}$.  %We say that a biset ${\sf X}$ of $V$ is {\it trivial} if $X_I = \emptyset$ or $X_O=V$, {\it nontrivial} otherwise.

Given a capacitated graph $(G=(V,E),c)$ and a positive integer $k$, we define a function ${\boldmath f}$ on the bisets on $V$ by {\boldmath $f({\sf X})$}$=k|w({\sf X})|+c(\delta_G(X_I,V-X_O))$. Observe that $f({\sf X})=f({\sf \overline{X}})$. This function will play a crucial role throughout the article.

We can now rephrase the definition of $(2,k)$-connectivity in terms of bisets. A capacitated graph $(G=(V,E),c)$ with $|V|\geq 3$ is $(2,k)$-connected if for every biset ${\sf X}$ which is nontrivial with respect to $V$, we have $f({\sf X})\geq 2k$. We also need the following slightly more advanced notion:  A capacitated graph $(H=(V+s,E),c)$ with $|V|\geq 3$ is called {\it $(2,k)$-connected in $V$} if for every biset ${\sf X}$ on $V$ which is nontrivial with respect to $V$, we have $f({\sf X})\geq 2k$. 
Observe that in $H=(V+s,E)$ the vertex $s$ belongs to $\overline{X}$ for any $X \subseteq V$.

\subsection{Splitting off}

Let $(H=(V+s,E),c)$ be a capacitated graph. For $v \in N_H(s)$ and a nonnegative integer $\alpha \leq c(sv)$, we denote by {\boldmath $(H,c)_v^{\alpha}$} the capacitated graph obtained from $(H,c)$ by decreasing the capacity of $sv$ by $\alpha$. If $c(sv)=0$ after the operation, we delete $sv$ from $H$.
For $(H,c)$ that is $(2,k)$-connected in $V$, we denote {\boldmath$U_{(H,c)}$}$=\{v \in V|\text{ }(H,c)_v^{1}\text{ is $(2,k)$-connected in $V$}\}$, a set that will play a significant role later on.
For a vertex $v \in N_H(s)$, we denote by {\boldmath $(H,c)^{max}_{v}$} the capacitated graph $(H,c)_v^{\alpha}$ where $\alpha$ is the maximum integer such that $(H,c)_v^{\alpha}$ is well-defined and $(2,k)$-connected in $V$.

For $u,v \in N_H(s)$ and a positive integer $\alpha\leq \min\{c(su),c(sv)\}$, we denote by {\boldmath $(H,c)_{u,v}^{\alpha}$} the capacitated graph obtained from $(H,c)$ by decreasing $c(su)$ and $c(sv)$ by $\alpha$ and increasing $c(uv)$ by $\alpha$. We delete edges of capacity $0$ and create the edge $uv$ if it does not exist yet. We also delete the arising loop if $u=v$. We call this operation the {\it $\alpha$-multiple splitting off} of $su$ and $sv$ and say that this $\alpha$-multiple splitting off contains $su$ and $sv$. We abbreviate $1$-multiple splitting off to {\it  splitting off}.
Suppose that $(H,c)$ is $(2,k)$-connected in $V$. We say that a pair $(su,sv)$ is {\it admissible} if $(H,c)_{u,v}^{1}$ is $(2,k)$-connected in $V$.
For  $u,v \in N_H(s)$, let $\alpha$ be the maximum integer such that $(H,c)_{u,v}^{\alpha}$ is well-defined and $(2,k)$-connected in $V$. We call an $\alpha$-mutiple splitting off of $(su,sv)$ a {\it maximal splitting off} of $(su,sv)$ and denote {\boldmath $(H,c)^{max}_{u,v}$}$=(H,c)_{u,v}^{\alpha}$. Observe that every maximal splitting off can be viewed as a series of splitting offs.
%{\boldmath $(H,c)^{max}_{u,v}$} the capacitated graph $(H,c)_{u,v}^{\alpha}$ where $\alpha$ is the maximum integer such that $(H,c)_{u,v}^{\alpha}$ is well-defined and $(2,k)$-connected in $V$.

We next give an important characterization of admissible pairs in $H$. Given a pair $(su,sv)$, a biset ${\sf X}$ which is nontrivial with respect to $V$ with either $f({\sf X})\leq 2k+1$ and $u,v \in X_I$ or $f({\sf X})= 2k,u,v \in X_O$ and $\{u,v\}\cap X_I \neq \emptyset$ is said to {\it block} $(su,sv)$. The following result can be found as Lemma 3.1 in \cite{DgGSz} and will be frequently used.

\begin{Lemma}\label{block}
Given a capacitated graph $(H=(V+s,E),c)$ that is $(2,k)$-connected in $V$ and $u,v \in N_H(s)$, $(su,sv)$ is admissible if and only if there is no biset blocking it.
\end{Lemma}

A biset that blocks a pair of edges $(su,sv)$ with $u \neq v$ is called {\it horrifying}. Note that the wall of a horrifying biset contains at most one vertex.  Further, observe that we can check whether a given biset is horrifying in $O(m)$ time by applying the definition.

While the following result is not explicitly proven in \cite{DgGSz}, its proof is almost literally the same as the one of Lemma 3.4 in \cite{DgGSz}. We therefore omit it. The result nevertheless plays a key role in our algorithm.

\begin{Lemma}\label{new34}
Let $(H=(V+s,E),c)$ be a capacitated graph that is $(2,k)$-connected in $V$ with $c(\delta_H(s))$ even. Let ${\sf X}$ be a horrifying biset, $u \in X_I \cap N_H(s)$ and $v \in N_H(s)-X_I$. If a biset ${\sf Y}$ blocks $(su,sv)$, then either ${\sf X} \cup {\sf Y}$ is horrifying or ${\sf X}$ and ${\sf Y}$ have the same wall of size $1$.
\end{Lemma}

We also require the following result that can be found in \cite{DgGSz} as Proposition 3.2.
\begin{Lemma}\label{dehors}
Let $(H=(V+s,E),c)$ be a capacitated graph that is $(2,k)$-connected in $V$ and let ${\sf X}$ be a horrifying biset. Then $N_H(s)-X_O \neq \emptyset$.
\end{Lemma}

%The following auxiliary result is new.
%\begin{Lemma}\label{neu}
%Let $(H=(V+s,E),c)$ be a capacitated graph that is $(2,k)$-connected in $V$ for some $k \geq 2$  with $c(\delta_H(s))$ even not containing an obstacle. Further, let $\A,\B$ be bisets on $V$ such that $\A$ is horrifying, $f(\B)=2k,B_O \neq V$, $\A$ and $\B$ have the same wall of size $1$ and $A_I \cap B_I \cap \Gamma_H(s)\neq \emptyset$. Then $\A \cup \B$ is horrifying.
%\end{Lemma}
%\begin{proof}
%As $\A$ is horrifying, we have $V-(\A \cap \B)_O \supseteq V-A_O\neq \emptyset$. As $(\A \cap \B)_I \neq \emptyset$ by assumption, we obtain that $\A \cap \B$ is non-trivial. As $(H,c)$ is $(2,k)$-connected in $V$, we obtain that $f(\A \cap \B)\geq 2k$. As $f(\A)\leq 2k+1$ and $f(\B)=2k$, we may now apply Proposition \ref{subm} and obtain $f(\A \cup \B)\leq f(\A)+f(\B)-f(\A \cap \B)\leq (2k+1)+2k-2k=2k+1$.

%It remains to show that $(\A \cup \B)_O \neq V$. Let $u \in A_I \cap B_I \cap \Gamma_H(s)$. As $c(\delta_H(s))$ is even and $(H,c)$ is $(2,k)$-connected in $V$ not containing an obstacle, by Theorem \ref{splitcomplete}, $(H,c)$ has a complete admissible splitting off. In particular, there is some $x\in V$ such that $(H,c)_{u,x}^1$ is $(2,k)$-connected in $V$. As $u \in B_I$,$f(\B)=2k$ and $B_O \neq V$, we have $x \notin B_O$. As $u \in A_I$ and $\A$ is horrifying, we have $x \notin A_I$. As $\A$ and $\B$ have the same wall of size $1$, we obtain $t \notin A_I \cup B_O=(\A \cup \B)_O$, so $(\A \cup \B)_O \neq V$.
%\qed \end{proof}

 A series of  splitting offs at $s$ that results in a capacitated graph in which $s$ is an isolated vertex is called a {\it complete splitting off} of $(H,c)$. It is easy to see that a complete splitting off exists if and only if $c(\delta_H(s))$ is even. A complete splitting off is {\it admissible} if each of the  splitting offs it contains is admissible in the current capacitated graph when being chosen. This is equivalent to the finally obtained capacitated graph being $(2,k)$-connected after deleting $s$. Finding such a complete admissible splitting off is the main difficulty in our algorithm. We strongly rely on a characterization of capacitated graphs having a complete admissible splitting off that can be found in \cite{DgGSz}. Before stating it, we need the following definition:

Let $(H=(V+s,E),c)$ be a capacitated graph that is $(2,k)$-connected in $V$ and with $c(\delta_H(s))$ even. An {\it obstacle} is a collection $\mathcal{B}$ of bisets in $V$ and a vertex $t \in N_H(s)$ satisfying the following:
\begin{eqnarray}
&& c(st) \text{ is odd and } t\in U_{(H,c),}\label{propt}\\
&& w({\sf B})=\{t\} \text { and } f({\sf B})=2k \text{ for all } {\sf B}\in \mathcal{B},\label{propB}\\
&& B_I \cap B'_I=\emptyset \text{ for all distinct } {\sf B},{\sf B}' \in \mathcal{B},\label{disjointB}\\
&& N_H(s)- \{t\}\subseteq \bigcup_{{\sf B} \in \mathcal{B}}B_I.\label{covering}
\end{eqnarray} 

We say that $t$ is the {\it special} vertex of the obstacle. It is easy to see that a capacitated graph containing an obstacle does not have a complete admissible splitting off, as every  splitting off of $(st,su)$ for some $u \in N_H(s)$ is blocked by the biset ${\sf B}\in \mathcal{B}$ with $u \in B_I$ and by Lemma \ref{block}. In \cite{DgGSz}, it is proved that the converse is also true:

\begin{Theorem}\label{splitcomplete}
Let $(H=(V+s,E),c)$ be a capacitated graph that is $(2,k)$-connected in $V$ for some $k \geq 2$ with $c(\delta_H(s))$ even. Then $(H,c)$ has a complete admissible splitting off at $s$ if and only if $(H,c)$ does not contain an obstacle.
\end{Theorem}

%If $H$ is $(2,k)$-connected and does not contain an obstacle, we say that the splitting off of $(su,sv)$ is admissible if $split((H,c),(su,sv),1)$ is $(2,k)$-connected. We need one more result from \cite{DgGSz}.

%The following property can be found in \cite{DgGSz} in Proposition 3.10.

\subsection{Basic algorithms}

In this section, we show that we can efficiently compute the maximum decrease of the capacity of an edge and the maximum multiplicity of a splitting off of an edge pair that maintain certain connectivity requirements. We first show this for the case of edge-connectivity and then apply this for the case of $(2,k)$-connectivity. All of these results are simple consequences of Lemma \ref{japonais}.

\begin{Lemma}\label{delete}
Given a capacitated graph $(H=(V+s,E),c)$ that is $k$-edge-connected in $V$ and a vertex $v \in N_H(s)$, we can compute the maximal $\alpha$ such that $(H',c')=(H,c)_v^{\alpha}$ is $k$-edge-connected in $V$ in $O(n(m+n \textrm{ }\log\textrm{ } n))$ time. Further, if $c'(sv)\neq 0$, we can compute a set $S$ with $v \in S \subsetneq V$ and $c'(\delta_{H'}(S))=k$ in $O(n(m+n \textrm{ }\log\textrm{ } n))$ time.
\end{Lemma}

\begin{proof}
Let $\gamma:=c(sv)$ and $(H_{\gamma},c_{\gamma})=(H,c)_{v}^{\gamma}$. Obviously we have $\alpha \leq \gamma$. The other condition $\alpha$ needs to satisfy is that $c'(\delta_{H'}(X))\geq k$ for every $\emptyset\neq X \subsetneq V$. If $v\notin X$, we obtain $c'(\delta_{H'}(X))=c(\delta_{H}(X))\geq k$. If $v \in X$, the condition is satisfied if and only if $0 \geq k-c'(\delta_{H'}(X))=k-(c_{\gamma}(\delta_{H_{\gamma}}(X))-\alpha+\gamma)=\alpha-(\gamma-k+c_{\gamma}(\delta_{H_{\gamma}}(X)))$. It follows that $\alpha=\min\{\gamma,\gamma-k+\lambda_{(H_{\gamma},c_{\gamma})}(V)\}$.  By Lemma \ref{japonais}, we can compute $\lambda_{(H_{\gamma},c_{\gamma})}(V)$, and hence $\alpha$, and also a set $S \subsetneq V$ with $c_{\gamma}(\delta_{H_{\gamma}}(S))=\lambda_{(H_{\gamma},c_{\gamma})}(V)$ in $O(n(m+n \textrm{ }\log\textrm{ } n))$ time. If $c'(sv)\neq 0$, we have $v \in S$ and $c'(\delta_{H'}(S))=k$.
\qed \end{proof}

%Given a capacitated graph $(H,c)=((V+s,E),c)$ and $u,v \in V$, we denote by $split(H,c,su,sv,\alpha)$ the capacitated graph obtained from $(H,c)$ by an $\alpha$-multiple splitting off of $(su,sv)$.

\begin{Lemma}\label{split}
Given a capacitated graph $(H=(V+s,E),c)$ that is $k$-edge-connected in $V$ and vertices $u,v \in N_H(s)$, we can compute the maximal $\alpha$ such that $(H',c')=(H,c)_{u,v}^{\alpha}$ is $k$-edge-connected in $V$ in $O(n(m+n \textrm{ }\log\textrm{ } n))$ time. Further, if $c'(su),c'(sv)\neq 0$, we can compute a set $S$ with $u,v \in S \subsetneq V$ and $c'(\delta_{H'}(S))\leq k+1$ in $O(n(m+n \textrm{ }\log\textrm{ } n))$ time.
\end{Lemma}

\begin{proof}
Let $\gamma=\min \{c(su),c(sv)\}$ and $(H_{\gamma},c_{\gamma})=(H,c)_{u,v}^{\gamma}$. Obviously we have $\alpha \leq \gamma$. The other condition $\alpha$ needs to satisfy is that $c'(\delta_{H'}(X))\geq k$ for every $\emptyset\neq X \subsetneq V$.  If $\{u,v\}- X\neq \emptyset$, we obtain $c'(\delta_{H'}(X))=c(\delta_{H}(X))\geq k$.  If $u,v \in X$, the condition is satisfied if and only if $0 \geq k-c'(\delta_{H'}(X))=k-(c_{\gamma}(\delta_{H_\gamma}(X))-2\alpha+2\gamma)=2(\alpha-(\gamma+\frac{1}{2}c_{\gamma}(\delta_{H_\gamma}(X))-\frac{1}{2}k))$. It follows that $\alpha=\min\{\gamma,\lfloor \gamma+\frac{1}{2}\lambda_{(H_{\gamma},c_{\gamma})}(V)-\frac{1}{2} k\rfloor \}$.  By Lemma \ref{japonais}, we can compute $\lambda_{(H_{\gamma},c_{\gamma})}(V)$, and hence $\alpha$,  and also a set $S \subsetneq V$ with $c_{\gamma}(\delta_{H_{\gamma}}(S))=\lambda_{(H_{\gamma},c_{\gamma})}(V)$, in $O(n(m+n \textrm{ }\log\textrm{ } n))$ time. If $c'(su),c'(sv)\neq 0$, we have $u,v \in S$ and $c'(\delta_{H'}(S))\leq k+1$.
\qed \end{proof}

\begin{Lemma}\label{bigdelete}
Given a capacitated graph $(H=(V+s,E),c)$ that is $(2,k)$-connected in $V$ and a vertex $v \in N_H(s)$, we can compute $(H,c)^{max}_{v}$ in $O(n^2(m+n \textrm{ }\log\textrm{ } n))$ time.% Further, if $c'(sv)\neq 0$, we can compute a nontrivial biset ${\sf X}$ with $v \in X_I$ and $f({\sf X})=2k$ in $O(n^2(m+n \textrm{ }\log\textrm{ } n))$ time.
\end{Lemma}

\begin{proof}By definition, $(H,c)_{v}^{max}=(H,c)_{v}^{\alpha}$ where $\alpha$ is the maximum integer such that $(H,c)_{v}^{\alpha}$ is $2k$-edge-connected in $V$ and $(H,c)_{v}^{\alpha}-x$ is $k$-edge-connected in $V-x$ for all $x\in V$. We first compute the maximum integer $\alpha '$ such that $(H,c)_{v}^{\alpha'}$ is $2k$-edge-connected in $V$. Using Lemma \ref{delete}, this can be done in $O(n(m+n \textrm{ }\log\textrm{ } n))$ time. Next observe that for any nonnegative integer $\beta$, $(H,c)_v^{\beta}-v=(H,c)-v$ is always $k$-edge-connected in $V-v$ by assumption. Now consider $x\in V-v$ and observe that for any nonnegative integer $\beta$, we have $(H,c)_v^{\beta}-x=(H-x,c)_v^{\beta}$. It follows from Lemma \ref{delete} that we can compute the maximum integer $\alpha_x$ such that $(H,c)_{v}^{\alpha_x}-x$ is $k$-edge-connected in $V-x$ in $O(n(m+n \textrm{ }\log\textrm{ } n))$ time. We now can compute $\alpha=\min\{\alpha',\min_{x \in V-v}\alpha_x\}$. The overall running time is $O(n^2(m+n \textrm{ }\log\textrm{ } n))$. 
%Now suppose that $c'(sv)\neq 0$. 
%If $\alpha=\alpha'$, that is $(H',c')=(H,c)_{v}^{\alpha'}$, then, 
 %by Lemma \ref{delete}, a set $S\subsetneq V$ with $v \in S$ and $c'(\delta_{H'}(S)) =2k$ can be found in $O(n(m+n \textrm{ }\log\textrm{ } n))$ time. We obtain that $(S,S)$ is a  biset with the desired properties. 
%If $\alpha=\alpha_x$ for some $x\in V-v$, that is $(H',c')-x=(H-x,c)_{v}^{\alpha_x}$, then, 
% by Lemma \ref{delete}, a set $S\subsetneq V-x$ with $v \in S$ and $c'(\delta_{H'-x}(S)) =k$ can be found in $O(n(m+n \textrm{ }\log\textrm{ } n))$ time. We obtain that $(S \cup \{x\},S)$ is a  biset with the desired properties.
\qed \end{proof}

\begin{Lemma}\label{bigsplit}
Given a capacitated graph $(H=(V+s,E),c)$ that is $(2,k)$-connected in $V$ and vertices $u,v \in N_H(s)$, we can compute $(H',c')=(H,c)^{max}_{u,v}$ in $O(n^2(m+n \textrm{ }\log\textrm{ } n))$ time. Further, if $c'(su),c'(sv)\neq 0$, we can compute a biset blocking $(su,sv)$ in $(H',c')$ in $O(n^2(m+n \textrm{ }\log\textrm{ } n))$ time.
\end{Lemma}

\begin{proof}
By definition, $(H',c')=(H,c)_{u,v}^{\alpha}$ where $\alpha$ is the maximum integer such that $(H,c)_{u,v}^{\alpha}$ is $2k$-edge-connected in $V$ and $(H,c)_{u,v}^{\alpha}-x$ is $k$-edge-connected in $V-x$ for all $x\in V$.
We first compute the maximum integer $\alpha '$ such that $(H,c)_{u,v}^{\alpha'}$ is $2k$-edge-connected in $V$. Using Lemma \ref{split}, this can be done in $O(n(m+n \textrm{ }\log\textrm{ } n))$ time.  For $x\in V-\{u,v\}$, 
%and observe that for any nonnegative integer $\beta$, we have $(H,c)_{u,v}^{\beta}-x=(H-x,c)_{u,v}^{\beta}$. 
  we can compute, by Lemma \ref{split}, the maximum integer $\alpha_x$ such that $(H,c)_{u,v}^{\alpha_x}-x=(H-x,c)_{u,v}^{\alpha_x}$ is $k$-edge-connected in $V-x$ in $O(n(m+n \textrm{ }\log\textrm{ } n))$ time. 
%Finally, observe that for any nonnegative integer $\beta$, $(H,c)_{v}^{\beta}-u=(H-u,c)_{v}^{\beta}$.  
We can compute, by Lemma \ref{delete}, the maximum integer $\alpha_u$ such that $(H,c)_{u,v}^{\alpha_u}-u=(H-u,c)_{v}^{\alpha_u}$ is $k$-edge-connected in $V-u$ in $O(n(m+n \textrm{ }\log\textrm{ } n))$ time. We similarly compute $\alpha_v$.
We now can compute $\alpha=\min\{\alpha',\min_{x\in V}\alpha_x\}$. The overall running time is $O(n^2(m+n \textrm{ }\log\textrm{ } n))$.

Now suppose that $c'(su),c'(sv)\neq 0$. 
If $\alpha=\alpha'$, that is $(H',c')=(H,c)_{u,v}^{\alpha'}$, then, by Lemma \ref{split}, a set $S\subsetneq V$ with $u,v \in S$ and $c'(\delta_{H'}(S))\leq 2k+1$ can be found in $O(n(m+n \textrm{ }\log\textrm{ } n))$ time. We obtain that $(S,S)$ is a biset blocking $(su,sv)$ in $(H',c')$. 
If $\alpha=\alpha_x$ for some $x\in V-\{u,v\}$, that is $(H',c')-x=(H-x,c)_{u,v}^{\alpha_x}$, then, by Lemma \ref{split}, a set $S\subsetneq V-x$ with $u,v \in S$ and $c'(\delta_{H'-x}(S))\leq k+1$ can be found in $O(n(m+n \textrm{ }\log\textrm{ } n))$ time. We obtain that $(S\cup \{x\},S)$ is a biset blocking $(su,sv)$ in $(H',c')$. 
Finally, if $\alpha=\alpha_u$ or $\alpha_v,$ say $\alpha_u,$ that is $(H',c')-u=(H-u,c)_{v}^{\alpha_u}$, then, by Lemma \ref{delete}, a set $S\subsetneq V-u$ with $v \in S$ and $c'(\delta_{H'}(S))=k$ can be found in $O(n(m+n \textrm{ }\log\textrm{ } n))$ time. We obtain that $(S\cup \{u\},S)$ is a biset blocking $(su,sv)$ in $(H',c')$.
\qed \end{proof}

%Further, if $c_{maxsplit(H,c,su,sv)}(su),c_{maxsplit(H,c,su,sv)}(sv)\neq 0$, we can compute a biset blocking $(su,sv)$ in $maxsplit(H,c,su,sv)$ in $O(n^2(m+n \textrm{ }\log\textrm{ } n))$.

\section{Minimal even extensions for $(2,k)$-connectivity}\label{ext}

A minimal even extension for $(2,k)$-connectivity of a capacitated graph $(G=(V,E_0),c_0)$ is obtained by adding a new vertex and edges incident to this vertex so that the obtained capacitated graph becomes $(2,k)$-connected in $V$ and so that the total capacity of the new edges is even and minimal. The importance of minimal even extensions for $(2,k)$-connectivity is due to a theorem from \cite{DgGSz} that shows that minimum augmentations for $(2,k)$-connectivity can be computed from minimal even extensions for $(2,k)$-connectivity by a complete admissible splitting off. We first give a simple algorithm to compute minimal even extensions for $(2,k)$-connectivity and give some basic properties.
%In this section, we discuss minimal even extensions. These are objects created by the first part of our algorithm which is also given here. We then cite a result from \cite{DgGSz} which makes it possible to compute a minimum augmentation from a minimal even extension via splitting offs.
We further show a property of minimal even extensions for $(2,k)$-connectivity which is essential to our splitting off algorithms. This then allows us to give a naive algorithm for the $(2,k)$-connectivity augmentation problem which is slower than the algorithm which is the main result of this article and is given later.

We first introduce the algorithm for computing a minimal even extension for $(2,k)$-connectivity. In order to avoid a technical definition whose details are not essential to this article, Algorithm \ref{algo1} will also serve as a definition for minimal even extensions for $(2,k)$-connectivity. Algorithm \ref{algo1} takes a capacitated graph $(G=(V,E_0),c_0)$ as input and adds a new vertex $s$ as well as edges of sufficiently high capacity between $s$ and all other vertices to make the capacitated graph $(2,k)$-connected in $V$. It then reduces these capacities in a greedy way as much as possible while maintaining $(2,k)$-connectivity in $V$. Finally, if the degree of $s$ is odd, it augments the capacity of a certain chosen edge by $1$. 

\begin{algorithm}
\caption {Minimal even extensions for $(2,k)$-connectivity}
\begin{algorithm2e}[H]\label{algo1}
%\begin{algorithmic}
%\medskip
\SetKwInOut{Input}{Input} 
\Input {A capacitated graph $(G=(V,E_0),c_0)$, an integer $k \geq 2$.}
\SetKwInOut{Output}{Output} 
\Output {A minimal even extension for $(2,k)$-connectivity of $(G,c_0)$.}
Create $(H,c)$ by adding a vertex $s$ to $V$ and adding an edge of capacity $2k$ between $s$ and every $v \in V$\;
Let $(v_1,\ldots,v_n)$ be an arbitrary ordering of the vertices of $V$\;
\For{$i=1,\ldots,n$}
	{$(H,c)=(H,c)^{max}_{v_i}$\;}
\If{$c(\delta_{H}(s))$ is odd}
	{choose the maximum $i^*$ such that $c(sv_{i^*})$ is odd\;
	$c(sv_{i^*})=c(sv_{i^*})+1$\;}
Return $(H,c)$ \;
%Create $(H_0,c_0)$ by adding a vertex $s$ to $V$ and adding an edge of capacity $2k$ between $s$ and every $v \in V$\;
%Let $(v_1,\ldots,v_n)$ be an arbitrary ordering of the vertices of $V$\;
%\For{$i=1,\ldots,n$}
%	{$(H_i,c_i)=(H_{i-1},c_{i-1})^{max}_{v_i}$\;}
%\If{$c_n(\delta_{H_n}(s))$ is odd}
%	{choose the maximum $i^*$ such that $c_n(sv_{i^*})$ is odd\;
%	$c_n(sv_{i^*})=c_n(sv_{i^*})+1$\;}
%Return $(H_n,c_n)$ \;
\end{algorithm2e}
\end{algorithm}

Given a capacitated graph $(G,c_0)$, a capacitated graph $(H,c)$ which is obtained by applying Algorithm \ref{algo1} to $(G,c_0)$ is called a {\it minimal even extension for $(2,k)$-connectivity} of $(G,c_0)$.

\begin{Proposition}\label{minevext}
Given a capacitated graph $(G,c_0)$, a minimal even extension for $(2,k)$-connectivity of $(G,c_0)$ can be computed in $O(n^3(m+n \textrm{ }\log\textrm{ } n))$ time.
\end{Proposition}

\begin{proof}
By definition, a minimal even extension for $(2,k)$-connectivity can be computed by Algorithm \ref{algo1}. It follows from Lemma \ref{bigdelete} that line 4  can be executed in $O(n^2(m+n \textrm{ }\log\textrm{ } n))$ time. As line 4 is executed $n$ times and the rest of Algorithm \ref{algo1} can be executed efficiently, Algorithm \ref{algo1} runs in $O(n^3(m+n \textrm{ }\log\textrm{ } n))$ time.
\qed \end{proof}

We next collect some basic properties of minimal even extensions for $(2,k)$-connectivity.

\begin{Proposition}\label{minevext}
Let $(H=(V+s,E),c)$ be a minimal even extension for $(2,k)$-connectivity of a capacitated graph $(G=(V,E_0),c_0)$. Then the following hold:
\begin{enumerate}[(a)]
\item $(H,c)$ is $(2,k)$-connected in $V$,
\item $c(\delta_H(s))$ is even,
\item $c(sv)$ is even for all $v \in U_{(H,c)}$,
\item $(H,c)_v^2$ is not $(2,k)$-connected in $V$ for any $v \in V$.
\end{enumerate}
\end{Proposition}

\begin{proof}
We obtain $(a)$ as an immediate consequence of the construction of Algorithm \ref{algo1}.

If the if-condition in line 5 is not satisfied, $c(\delta_H(s))$ is even and remains unchanged in the rest of the algorithm. Otherwise, $c(\delta_H(s))$ is odd and is augmented by $1$ in line 7. This yields $(b)$.

 We denote by $(H_i,c_i)$ the capacitated graph defined in line 4 in iteration $i$. Since line 4 is executed, for all $i \in \{1,\ldots,n\}$ with $c_i(sv_i)\ge 1$, there exists a biset ${\sf X}^i$ such that $v_i\in X^i_I$ and $f_{(H_i,c_i)}({\sf X}^i)=2k.$
If the if-condition in line 5 is not satisfied, then $f_{(H,c)}({\sf X}^{i})\leq f_{(H^{i},c^{i})}({\sf X}^{i})=2k$ for all $i \in \{1,\ldots,n\}$  with $c_i(sv_i)\ge 1$ and so $U_{(H,c)}=\emptyset$. If the if-condition in line 5 is satisfied, then $i^*$ is defined in line 6. For all $i \in \{1,\ldots,i^*-1\}$  with $c_i(sv_i)\ge 1$, since $c_i(sv_i)\ge 1, c_i(sv_{i^*})=2k$ and $f_{(H_i,c_i)}({\sf X}^i)=2k,$ we have $v_{i^*}\notin X^i_I.$ This yields $f_{(H,c)}({\sf X}^{i})\leq f_{(H^{i},c^{i})}({\sf X}^{i})=2k$, so $v_i \notin U_{(H,c)}$. For all $i\in \{i^*,\ldots,n\}$, $c(sv_i)$ is even after the execution of line 7. This proves $(c)$.

For any $v_i \in V$ with $c(sv_i)\ge 1$, we have $f_{(H,c)}({\sf X}^{i})\leq f_{(H^{i},c^{i})}({\sf X}^{i})+1=2k+1$. This proves $(d)$.
\qed \end{proof}

We now give the following theorem which is the reason for us considering minimal even extensions for $(2,k)$-connectivity. Its proof can be found in \cite{DgGSz}. It shows that minimum augmentations for $(2,k)$-connectivity can be computed from minimal even extensions for $(2,k)$-connectivity if we can find a complete admissible splitting off at $s$.

\begin{Theorem}\label{twice}
Let $(G=(V,E_0),c_0)$ be a capacitated graph, $(H=(V+s,E),c)$ a minimal even extension for $(2,k)$-connectivity of $(G,c_0)$ and let $(H',c')$ be obtained from $(H,c)$ by a complete admissible splitting off. Then $(H',c')-s$ is a minimum $(2,k)$-connected augmentation of $(G,c)$. 
\end{Theorem}

We now prove an important result that shows that when finding a complete admissible splitting off, we do not need to worry about obstacles.

\begin{Lemma}\label{antiobstacle}
Let $(H=(V+s,E),c)$ be a minimal even extension for $(2,k)$-connectivity of a capacitated graph $(G,c_0)$ and let $(H_1,c_1)$ be obtained from $(H,c)$ by a series of admissible  splitting offs. Then $(H_1,c_1)$ contains no obstacle.
\end{Lemma}

\begin{proof}
Suppose that $(H_1,c_1)$ contains an obstacle $\mathcal{B}$ with special vertex $t$. By  \eqref{propt}, $c_1(st)$ is odd and $t\in U_{(H_1,c_1)}\subseteq U_{(H,c)}$. It follows by Proposition \ref{minevext}$(c)$ that $c(st)$ is even, so $st$ was split off with an edge $sv$. By Proposition \ref{minevext}$(d)$, we have $v \neq t$. Let $(H_2,c_2)$ be the capacitated graph such that $(H_1,c_1)=(H_2,c_2)_{t,v}^1$.

\begin{Claim}\label{noob}
$(H_2,c_2)$  contains no obstacle.
\end{Claim}

\begin{proof} Suppose otherwise, so $(H_2,c_2)$ contains an obstacle $\mathcal{B}'$. By \eqref{propt}, the special vertex of $\mathcal{B}'$ cannot be $t$ because $c_2(st)=c_1(st)+1$ is even. It follows by \eqref{covering} that $t \in B'_I$ for some ${\sf B'}\in \mathcal{B}'$. Now  \eqref{propB} yields $f_{(H_2,c_2)}({\sf B}')=2k$ which contradicts that by \eqref{propt}, $t \in U_{(H_1,c_1)}\subseteq U_{(H_2,c_2)}$.
\qed 
\end{proof}

By Claim \ref{noob} and Theorem \ref{splitcomplete}, $(H_2,c_2)$  has a complete admissible splitting off. In particular, as $c_2(st)\geq 2$, there exist $x,y\in N_{H_2}(s)$ (possibly $x=y$) such that  $(H_3,c_3)=((H_2,c_2)_{t,x}^1)_{t,y}^1$ is $(2,k)$-connected in $V.$ 
By Proposition \ref{minevext}$(d)$, $(H,c)_{t,t}^1=(H,c)_t^2$ is not $(2,k)$-connected in $V$ and hence neither is $(H_2,c_2)_{t,t}^1$, so $x,y \neq t$.
Obviously $x,y \neq v$, so, by \eqref{covering}, $x,y \in \bigcup_{{\sf B}\in \mathcal{B}} B_I$.  Then, by \eqref{propB}, we have 
$2+|\mathcal{B}|k\le 2+\sum_{{\sf B}\in \mathcal{B}}c_3(\delta_{H_3-t}(B_I))=\sum_{{\sf B}\in \mathcal{B}}c_2(\delta_{H_2-t}(B_I)) = 1+\sum_{{\sf B}\in \mathcal{B}}c_1(\delta_{H_1-t}(B_I))=1+|\mathcal{B}|k$, a contradiction.
\qed \end{proof}

We are now ready to give a first naive algorithm for finding a complete admissible splitting off of minimal even extensions for $(2,k)$-connectivity.
\medskip

\begin{algorithm}
\caption {Naive splitting off}
\begin{algorithm2e}[H]\label{algox}
%\begin{algorithmic}
%\medskip
\SetKwInOut{Input}{Input} 
\Input {A minimal even extension for $(2,k)$-connectivity $(H=(V+s,E),c)$ of a capacitated graph $(G,c_0).$}
\SetKwInOut{Output}{Output} 
\Output {A minimum $(2,k)$-connected augmentation of $(G,c_0).$}

\For{$u\neq v \in N_H(s)$}
	{$(H,c)=(H,c)^{max}_{u,v}$\;}
Return $(H,c)-s$ \;
%Create $(H_0,c_0)$ by adding a vertex $s$ to $V$ and adding an edge of capacity $2k$ between $s$ and every $v \in V$\;
%Let $(v_1,\ldots,v_n)$ be an arbitrary ordering of the vertices of $V$\;
%\For{$i=1,\ldots,n$}
%	{$(H_i,c_i)=(H_{i-1},c_{i-1})^{max}_{v_i}$\;}
%\If{$c_n(\delta_{H_n}(s))$ is odd}
%	{choose the maximum $i^*$ such that $c_n(sv_{i^*})$ is odd\;
%	$c_n(sv_{i^*})=c_n(sv_{i^*})+1$\;}
%Return $(H_n,c_n)$ \;
\end{algorithm2e}
\end{algorithm}
\medskip

%$2+|\mathcal{B}|2k\le 2+\sum_{{\sf B}\in \mathcal{B}}f_{(H_3,c_3)}({\sf B})=\sum_{{\sf B}\in \mathcal{B}}f_{(H_2,c_2)}({\sf B})\le 1+\sum_{{\sf B}\in \mathcal{B}}f_{(H_1,c_1)}({\sf B})=1+|\mathcal{B}|2k$, a contradiction.
By Lemma \ref{antiobstacle}, no obstacle is created during the execution of Algorithm \ref{algox} and so, by Theorems \ref{splitcomplete} and \ref{twice}, the output of Algorithm \ref{algox} is a minimum $(2,k)$-connected augmentation of $(G,c_0).$ As line 2 is executed at most $n^2$ times and by Lemma \ref{bigsplit}, Algorithm \ref{algox} runs in $O(n^4(m+n \log n))$ time. Together with Algorithm \ref{algo1}, this yields an $O(n^4(m+n \log n))$ time algorithm for the $(2,k)$-connectivity augmentation problem.
\section{A fast splitting off algorithm}\label{algo}
This section is dedicated to refining Algorithm \ref{algox} in order to improve its running time from $O(n^4(m+n \log n))$ time to $O(n^3(m+n \log n))$ time. Together with Algorithm \ref{algo1} and Theorem \ref{twice}, this yields an $O(n^3(m+n \log n))$ time algorithm for the $(2,k)$-connectivity augmentation problem.

While Algorithm \ref{algox} executes maximal splitting offs for all pairs of edges incident to $s$ and therefore executes maximal splitting offs for a number of pairs which is quadratic in $n$ before terminating, the refined version, Algorithm \ref{algo}, carefully chooses the pairs to be split off. This allows us to terminate after a number of maximal splitting offs which is linear in $n$.

 In order to achieve this, Algorithm \ref{algo} maintains not only a capacitated graph that is obtained from the minimal even extension for $(2,k)$-connectivity that is its input by a series of splitting offs. It also stores the information obtained from the fact that certain pairs are not admissible in the form of a biset ${\sf X}$. 
If two edges incident to $s$ both have their second endvertex in the inner set of ${\sf X}$, their splitting off is not admissible. 
 The maintenance of ${\sf X}$ therefore allows us to avoid attempts of splitting offs of pairs which are known to be nonadmissible.

During each iteration of the algorithm, we execute one or two maximal splitting offs. If none of these maximal splitting offs delete an edge incident to $s$, we modify ${\sf X}$. The number of neighbors of $s$ which are not covered by $X_O$ never increases. Further, after a small constant number of iterations of our algorithm, either an edge incident to $s$ is deleted or the number of neighbors of $s$ not covered by $X_O$ decreases. This allows us to obtain the desired running time.

In the first part of this section, we show a key lemma that is needed to modify ${\sf X}$ in a favorable way. After, we describe the algorithm in the form of a pseudocode. Finally, we prove the correctness of the algorithm and analyze its running time.  

\subsection{Key lemma}\label{new}

This part gives a result that allows us to modify the biset ${\sf X}$.

\begin{Lemma}\label{vide}
Let $(H=(V+s,E),c)$ be a capacitated graph that is $(2,k)$-connected in $V$ for some $k \geq 2$  and has a complete admissible splitting off. Let ${\sf X}$  be a horrifying biset, $u \in X_I \cap N_H(s), v \in N_H(s)-X_O,$ $\Y$ a biset blocking $(su,sv)$, $z \in (X_I-Y_I)\cap N_H(s)$ and suppose that $\X \cup \Y$ is not horrifying. Let $\Z$ be a biset blocking $(sv,sz)$. Then $\X \cup \Y \cup \Z$ is horrifying.
\end{Lemma}

\begin{proof} 
%We first determine the value of $f$ for a number of bisets.
%\begin{Claim}\label{valeurf}
%$f(\X)=2k+1=f(\Y)$ and $f(\X-\Y)=2k=f(\Y-\X)$.
%\end{Claim}
%\begin{proof}
%As $w \in (\X-\Y)_I$ and $v \notin (\X-\Y)_O$, $\X-\Y$ is nontrivial. As $(H,c)$ is $(2,k)$-connected in $V$, we obtain $f(\X-\Y)\geq 2k$. Similarly, $f(\Y-\X)\geq 2k$. As $\X$ and $\Y$ are horrifying, we have $f(\X),f(\Y)\leq 2k+1$. As $u \in X_I \cap Y_I \cap \Gamma_H(s)$, the statements follow from Lemma \ref{intersectionpetite}.
%\qed \end{proof}
By Lemma \ref{new34}, $w(\X)=w(\Y)=\{p\}$ for some vertex $p\in V.$ Since $\{u\},\{v\} \neq w(\X)=w(\Y)$ and $\Y$ blocks $(su,sv)$, we have $u,v\in Y_I.$
Let {\boldmath$(H',c')$} $=(H,c)-p$ and {\boldmath$d'$}$(S):=c'(\delta_{H'}(S))$ for  $S\subseteq V-p$. For $S_1,S_2 \subseteq V-p$, we use {\boldmath $\overline{d'}(S_1,S_2)$} for $c'(\delta_{H'}(S_1\cap S_2, (V+s-p)-(S_1\cup S_2)))$. Observe that $(H',c')$ is $k$-edge-connected in $V$ since $(H,c)$ is $(2,k)$-connected in $V$. This yields that $d'(S)\geq k$ for any $\emptyset \neq S\subset V-p$.

We  distinguish two cases depending on where the wall of $\Z$ is located.
\begin{case}
$w(\Z)=\{p\}$.
\end{case}

We show that in this case $\X\cup\Y\subseteq\Z$ and hence  $\X \cup \Y \cup \Z=\Z$ is horrifying.
For the sake of a contradiction, suppose that $(X_I\cup Y_I)-Z_I\neq\emptyset.$
%Let us suppose for a contradiction that this is not the case. 
In order to use some symmetry arguments, let {\boldmath$A^1$} $=X_I,$ {\boldmath$A^2$} $=Y_I$ and {\boldmath$A^3$} $=Z_I$. Then, by $z\in (X_I\cup Z_I)-Y_I$, $v\in (Y_I\cup Z_I)-X_I,$ and the assumption that $(X_I \cup Y_I)-Z_I\neq \emptyset$, we obtain $(A^i\cup A^j)-A^{\ell}\not=\emptyset$ whenever $\{i,j,\ell\}=\{1,2,3\}$. 
%\end{equation}
%
Since $\X,\Y$ and $\Z$ are horrifying bisets whose wall is $\{p\}$, we have $d'(A^i)\le (2k+1)-k=k+1$ for $i \in \{1,2,3\}$.
%Since $(H,c)$ is $(2,k)$-connected in $V$ and $\X^i$ is horrifying, we have 
%%
%\begin{eqnarray}
%d'(A)&\ge &k \text{ for all } \emptyset\neq A\subset V-p, \label{kec}\\
%d'(X^i_I)&\le &k+1.\label{cut}
%\end{eqnarray}
%
%\begin{eqnarray}
%	X^i_I\cap Y^j_I	&	\neq	&	\emptyset,\label{inter}\\
%	\overline d'(X^i_I\cup X^j_I,X^k_I)	&	\ge	&	 2.\label{inter2}
%\end{eqnarray}
%
%\begin{Claim}\label{union}
%$d'(X^i_I\cup X^j_I)\le k+2.$
%\end{Claim}
%%
%\begin{proof}
%By \eqref{submod}, \eqref{cut}, \eqref{inter} and \eqref{kec}, we have 
%$d'(X^i_I\cup X^j_I)\le d'(X^i_I)+d'(X^j_I)-d'(X^i_I\cap X^j_I)\le (k+1)+(k+1)-k=k+2.$
%\end{proof}
%
\begin{Claim}\label{incl}
$A^{\ell}\subseteq A^i\cup A^j$ whenever  $\{i,j,\ell\}=\{1,2,3\}$. 
\end{Claim}
\begin{proof}
Suppose  that  $A^{\ell}-(A^i\cup A^j) \not=\emptyset$ for some $\{i,j,\ell\}=\{1,2,3\}$.
By $z\in X_I\cap Z_I\cap N_H(s), v\in Y_I\cap Z_I\cap N_H(s)$ and $u\in Y_I\cap X_I\cap N_H(s),$ we have $A^i\cap A^j\neq\emptyset$ and $\overline {d'}(A^{i}\cup A^j, A^{\ell})\geq |(A^{i}_I\cup A^j_I)\cap A^{\ell}_I \cap N_H(s)|\ge 2.$ As $A^i\cap A^j\neq\emptyset$ and $(A^i\cup A^j)-A^{\ell} \neq \emptyset$, it follows that $d'(A^i\cap A^j), d'((A^i\cup A^j)-A^{\ell}) \geq k$. 
%Further, $w\in X_I\cap Z_I, v\in Y_I\cap Z_I$ and $u\in Y_I\cap X_I$ yield 
Then, as $d'(A^{i}),d'(A^j),d'(A^{\ell})\leq k+1$, Proposition \ref{ftyig}$(a)$ and $(b)$ yield
%\eqref{cut}, \eqref{inter}, \eqref{kec}, \eqref{submod}, \eqref{submod-}, \eqref{unionprop}, \eqref{kec} and \eqref{inter2}, 
\begin{flalign*}
3(k+1)-k&\ge d'(A^i)+d'(A^j)-d'(A^i\cap A^j)+d'(A^{\ell})\\
&\ge d'(A^i\cup A^j)+d'(A^{\ell})\\
&=d'((A^i\cup A^j)-A^{\ell})+d'(A^{\ell}-(A^i\cup A^j))+2\overline {d'}(A^{i}\cup A^j, A^{\ell})\\
&\ge 2k+4,
\end{flalign*}
 a contradiction.
\qed \end{proof}
%
%\begin{Claim}\label{notempty}
%$X^i_I-X^j_I\neq\emptyset.$
%\end{Claim}
%
%\begin{proof}
%By \eqref{unionprop} and Claim \ref{incl}, we have $X^i_I-X^j_I=((X^i_I\cup Y^k_I)-X^j_I)-(X^k_I-(X^i_I\cup X^j_I))\neq\emptyset.$
%\end{proof}
%
\begin{Claim}\label{sort}
For any $i\neq j\in \{1,2,3\}$, we have $\overline {d'}(A^i,A^j)\le 1.$
\end{Claim}
\begin{proof}
Let $\ell$ be the remaining element in $\{1,2,3\}-\{i,j\}$. By assumption, $(A^i\cup A^{\ell})-A^j\neq \emptyset$. By Claim \ref{incl}, $A^{\ell}-(A^i\cup A^j)=\emptyset$.
This yields $A^i-A^j=((A^i\cup A^{\ell})-A^j)-(A^{\ell}-(A^i\cup A^j))\neq\emptyset.$ Similarly, $A^j-A^i\neq\emptyset.$ It follows that $d'(A^i-A^j), d'(A^j-A^i)\geq k$. As $d'(A_i),d'(A_j)\leq k+1$, Proposition \ref{ftyig}$(b)$ yields
$2\overline {d'}(A^i,A^j)=d'(A^i)+d'(A^j)-d'(A^i-A^j)-d'(A^j-A^i)\le 2(k+1)-2k=2.$
\qed \end{proof}
%
%\begin{Claim}\label{small}
%$d'(X_I\cup Y_I)\le 3.$
%\end{Claim}
%
%\begin{proof}
%By Claims \ref{incl} and \ref{sort}, we have $d'(X_I\cup Y_I)=d'(X_I\cup Y_I\cup Z_I)\le \sum_{i\neq j}\overline d'(X^i_I,X^j_I)\le 3.$
%\end{proof}
%
\begin{Claim}\label{notv}
$d'(X_I\cup Y_I)\ge k+2.$
\end{Claim}
\begin{proof}
%Since $\X\cup\Y$ is not horrifying, it is sufficient to prove that $(\X\cup\Y)_O\neq V.$ 
As $(H,c)$ has a complete admissible splitting off, there exist $x,y \in N_H(s)$ such that $(H'',c''):=((H,c)_{z,x}^1)_{v,y}^1$ is $(2,k)$-connected in $V$. Since $z \in X_I \cap Z_I, v \in Y_I \cap Z_I$ and $\X,\Y,\Z$ are horrifying, we obtain $x\notin X_I\cup Z_I$ and $y\notin Y_I\cup Z_I.$ Then, by Claim \ref{incl}, $x,y\notin X_I\cup Y_I.$ If $x=p=y$, then $f_{(H'',c'')}(\Z)=f(\Z)-2$, so, since $(H'',c'')$ is $(2,k)$-connected in $V$  and  $\Z$ is horrifying, we have $2k\le f_{(H'',c'')}(Z)=f(\Z)-2\le (2k+1)-2=2k-1,$ a contradiction. So one of $x$ and $y$ belongs to $V-(\X\cup\Y)_O$ and hence $\X \cup \Y$ is a nontrivial biset with respect to $V$. 
Then, since $\X\cup\Y$ is not horrifying, we have $d'(X_I\cup Y_I)=f(\X\cup\Y)-|w(\X\cup\Y)|k\ge (2k+2)-k=k+2.$
\qed \end{proof}
By Claim \ref{incl}, we have $X_I \cup Y_I=X_I \cup Y_I \cup Z_I$ and every edge that contributes to $d'(X_I\cup Y_I\cup Z_I)$ also contributes to $\overline{d'}(A^i,A^j)$ for some $i \neq j \in \{1,2,3\}$. By $k\ge 2,$  and Claims \ref{notv} and \ref{sort}, we obtain $4\le k+2\le d'(X_I\cup Y_I)=d'(X_I\cup Y_I\cup Z_I)\le \sum_{i\neq j}\overline {d'}(A^i,A^j)\le 3,$ a contradiction. This finishes the case. \qed
%\begin{case}
%$w(\Z)\neq \{p\}$
%\end{case}
%By symmmetry, we may suppose that $w(\Z)\neq \{v\}$.
%As $\X$ is horrifying, $u \in X_I, v\in \Gamma_H(s)-X_I,\Z$ blocks $(su,sv)$ and $w(\Z)\neq \{p\}$, we may apply Lemma \ref{new34} and obtain that $\X \cup \Z$ is horrifying. As $w(\Z)\neq \{v\}$, we obtain $u,v \in (\X \cup \Z)_I$ and so $\X \cup \Z$ blocks $(su,sv)$. As $\Y$ is horrifying, $v \in Y_I$, $u \in \Gamma_H(s)-Y_I$, we may apply Lemma \ref{new34} once more and obtain that either $\X \cup \Y \cup \Z$ is horrifying or $w(\X \cup \Z)=w(\Y)$. In the first case there is nothing to prove, so we may suppose that $w(\X\cup \Z)=w(\Y)=w(\Y-\X)$. By Claim \ref{valeurf}, we have $f(\Y-\X)=2k$. Also, as $w(\Z)\neq \{v\}$ we have $v \in (\Y-\X)_I\cap (\X \cup Z)_I\cap \Gamma_H(s)$, so  $(\Y-\X)_I\cap (\X \cup Z)_I\cap \Gamma_H(s)\neq \emptyset$. As $(\Y-\X)_O \neq V$, Lemma \ref{neu} yields that again $\X \cup \Y \cup \Z$ is horrifying. This finishes the proof.
%\qed \end{proof}

\begin{case}
$w(\Z)\neq \{p\}$.
\end{case}
By symmetry, we may suppose that $v\in Z_I$.
As $z\in X_I, v\in N_H(s)-X_I,\Z$ blocks $(sz,sv)$ and $w(\Z)\neq \{p\}$, we may apply Lemma \ref{new34} to obtain that $\X \cup \Z$ is horrifying. Since $z,v \in (\X \cup \Z)_I$, $\X \cup \Z$ blocks $(sz,sv)$. As $v \in Y_I$, $z\in N_H(s)-Y_I$, we may apply Lemma \ref{new34} once more and obtain that either $\X \cup \Y \cup \Z$ is horrifying or $w(\X \cup \Z)=w(\Y)$. In the first case we are done, so we may suppose that $w(\X\cup \Z)=w(\Y)=\{p\}$. If $Y_I-(X_I\cup Z_I)=\emptyset$, we obtain that $\X \cup \Y \cup \Z=\X\cup \Z$ is horrifying. 

We may therefore suppose that $Y_I-(X_I\cup Z_I)\neq\emptyset$. As $z \in (X_I \cup Z_I)-Y_I$, we obtain that $d'(Y_I-(X_I\cup Z_I))\geq k$ and $d'((X_I \cup Z_I)-Y_I)\geq k.$ Also, as $\X \cup \Z$ and $\Y$ are horrifying bisets whose wall is $\{p\}$, we obtain $d'(X_I \cup Z_I)\leq k+1$ and $d'(Y_I)\leq k+1$. As $u,v \in (X_I \cup Z_I)\cap Y_I \cap N_H(s)$, Proposition \ref{ftyig} $(b)$ yields  $2(k+1)\ge d'(X_I\cup Z_I)+d'(Y_I)=d'((X_I\cup Z_I)-Y_I)+d'(Y_I-(X_I \cup Z_I))+2\overline {d'}(X_I\cup Z_I,Y_I)\ge 2k+4,$ a contradiction. %Thus $\X \cup \Y \cup \Z=\X\cup \Z$ is horrifying that finishes the proof.
\qed \end{proof}

\subsection{Decription of the algorithm} \label{soff}We are now ready to describe the algorithm in the form of a pseudocode. It first is initialized with the input capacitated graph and an empty biset ${\sf X}$. The main part of the algorithm consists of  a while-loop in which maximal splitting offs are executed and ${\sf X}$ is modified. In order to apply the structure found in Lemma \ref{vide}, we need $\X$ to be horrifying. Therefore, the first part of the while-loop in lines 3 to 7 is concerned with reinitializing $\X$ with a horrifying biset if $\X$ is not horrifying before the iteration. The main part from line 9 to 24 deals with the case when $\X$ is horrifying. Algorithm \ref{algo} then performs up to two maximal splitting offs of pairs of edges incident to $s$ whose choice depends on $\X$. If none of these two maximal splitting offs leads to the deletion of an edge incident to $s$, Algorithm \ref{algo} augments $\X$ in a beneficial way. After the last iteration of the while-loop, Algorithm \ref{algo} outputs the obtained capacitated graph after deleting $s$.

%\begin{frame}[fragile]
\begin{algorithm}[H]
\caption {Complete admissible splitting off}
\begin{algorithm2e}[H]\label{algo}
%{\small
\SetKwInOut{Input}{Input} 
\Input {A minimal even extension for $(2,k)$-connectivity $(H=(V+s,E),c)$ of a capacitated graph $(G,c_0).$}
\SetKwInOut{Output}{Output} 
\Output {A minimum $(2,k)$-connected augmentation of $(G,c_0).$}

${\sf X} := (\emptyset,\emptyset)$\;
\While{$|N_H(s)|\geq 2$}
	{
		\If{$\X$ is not horrifying}
			{	let $u\neq v \in N_H(s)$\;
				$(H,c)=(H,c)^{max}_{u,v}$\;
				\If{$c(su),c(sv)>0$}
					{
					let {\sf X} be a biset blocking $(su,sv)$\;
					}
			}
%		\Else
%			{\If{${\sf Y}= (\emptyset,\emptyset)$}
%				{let $u \in X_I \cap \Gamma_H(s)$\;
%				let $v \in \Gamma_H(s)-X_O$\; 
%				$(H,c)=(H,c)^{max}_{u,v}$\;
%				\If{$c(su),c(sv)>0$}
%					{let ${\sf Y}$ be a biset blocking $(su,sv)$\;}
%				}
			\Else
				{	let $u \in X_I \cap N_H(s)$\;
					let $v \in N_H(s)-X_O$\;
					$(H,c)=(H,c)^{max}_{u,v}$\;
					\If{$c(su),c(sv)>0$}
						{let ${\sf Y}$ be a biset blocking $(su,sv)$\;
						\If{$\X \cup \Y$ is horrifying}	
							{$\X=\X \cup \Y$\;}		
						\Else
							{\If {$X_I \cap N_H(s)\subseteq Y_I$}
								{$\X = \Y$\;}
							\Else
								{let $z \in (X_I-Y_I)\cap N_H(s)$\;
								$(H,c)=(H,c)^{max}_{v,z}$\;
								\If {$c(sv),c(sz)>0$}
									{let $\Z$ be a biset blocking $(sv,sz)$\;
									$\X=\X \cup \Y \cup \Z$\;
									}
								}
							}		
						}
				}
	}
\Return $(H,c)-s$\;
\end{algorithm2e}
\end{algorithm}

\subsection{Analysis of the algorithm}

This last section is dedicated to the analysis of Algorithm \ref{algo}. We first give a collection of properties of the capacitated graphs and bisets obtained at intermediate steps of Algorithm \ref{algo}. We then conclude the correctness and the running time of Algorithm \ref{algo}.

\begin{Proposition}\label{ana}
The following hold for every iteration $i$ of the while loop starting in line 2:
\begin{enumerate}[(a)]
\item All steps in iteration $i$ are well-defined.
\item $(H,c)$ is $(2,k)$-connected in $V$ and has a complete admissible splitting off after iteration $i$.
\item $f(\X)\leq 2k+1,|w(\X)|\leq 1$ and $X_O \neq V$ after iteration $i$.
\end{enumerate}
\end{Proposition}
\begin{proof}
By Lemma \ref{antiobstacle}, $(b)$ holds before iteration $1$ and trivially $(c)$ also holds. Inductively, we may suppose that $(a),(b)$ and $(c)$ hold for all iterations $1,\ldots,i-1$. We show that they also hold for iteration $i$:

$(a)$: The choice of $u$ and $v$ in line 4 is justified by the fact that the while-condition in line 2 was satisfied. If $\X$ is horrifying, the choice of $u$ in line 9 is justified and Lemma \ref{dehors} justifies the choice of $v$ in line 10. The choice of $z$ in line 20 is justified by the fact that the if-condition in line 17 was not satisfied. The horrifying bisets in lines 7,13 and 23 exist by Lemma \ref{block}.

$(b)$: It follows immediately from the construction that $c(\delta_H(s))$ always remains even and $(H,c)$ always remains $(2,k)$-connected in $V$. By Lemma \ref{antiobstacle}, no obstacle in $(H,c)$ can ever be created. Now Theorem \ref{splitcomplete} yields that $(H,c)$ has a complete admissible splitting off after iteration $i$.

$(c)$: As splitting offs do not increase $f$, it suffices to prove that $\X$ either remains unchanged or is horrifying after iteration $i$. First suppose that the if-condition in line 3 is satisfied. If the if-condition in line 6 is not satisfied, $\X$ remains unchanged. Otherwise, $\X$ is replaced by a horrifying biset in line 7.

Now suppose that the else-case starting in line 8 is executed. If the if-condition in line 12 is not satisfied, $\X$ remains unchanged, so suppose otherwise. If the if-condition in line 14 is satisfied, $\X$ is replaced by a horrifying biset in line 15, so suppose otherwise. If the if-condition in line 17 is satisfied, $\X$ is replaced by a horrifying biset in line 18. So suppose that the else-case starting in line 19 is executed. If the if-condition in line 22 is not satisfied, $\X$ remains unchanged. Otherwise, $u \in X_I\cap N_H(s), v \in N_H(s)-X_O, \Y$ blocks $(su,sv), z \in (X_I-Y_I)\cap N_H(s), \Z$ blocks $(sv,sz)$ and $\X \cup \Y$ is not horrifying. Together with $(b)$, Lemma \ref{vide} yields that $\X \cup \Y \cup \Z$ is horrifying, so $\X$ is replaced by a horrifying biset in line 24.  
\qed \end{proof}

We now obtain the correctness of our algorithm as a simple corollary:
\begin{Theorem}
If Algorithm \ref{algo} terminates,  it outputs a minimum $(2,k)$-connected augmentation of $(G,c_0)$.
\end{Theorem}
\begin{proof}
By Theorem \ref{splitcomplete}, it is sufficient to prove that Algorithm \ref{algo} executes a complete admissible splitting off of the input capacitated graph.
Let $(H,c)$ be the current capacitated graph after the last iteration of the while-loop. By construction, $|N_H(s)|\leq 1$. If $N_H(s)$ contains a single vertex $u$, by Proposition \ref{minevext} $(d)$, $(su,su)$ is not admisible in the input capacitated graph. As $(H,c)$ has been obtained by admissible splitting offs, $(su,su)$ is neither admissible in $(H,c)$. It follows that $(H,c)$ does not have a complete admissible splitting, a contradiction to Proposition \ref{ana} $(b)$. Hence $s$ is an isolated vertex in $(H,c)$ and so, by  Proposition \ref{ana} $(b)$, Algorithm \ref{algo} executed a complete admissible splitting off.
\qed
\end{proof}

The remaining part is concerned with the running time analysis of Algorithm \ref{algo}.

\begin{Theorem}
Algorithm \ref{algo} runs in $O(n^3(m+n \log n))$ time.
\end{Theorem}
\begin{proof}
Obviously the initialization of Algorithm \ref{algo} and the final output can be executed efficiently. Also, it follows from Lemma \ref{bigsplit} and the fact that we can check in $O(m)$ time whether a given biset is horrifying that every iteration of the while-loop starting in line 2 can be executed in $O(n^2(m+n \log n))$ time. It remains to show that the while-loop runs a linear number of times. In order to do this, we define the parameter {\boldmath$M$} $=|N_H(s)|+|N_H(s)-X_O|$. The decrease of $M$ measures the progress of our algorithm. We next prove two claims that show that $M$ decreases regularly.

\begin{Claim}\label{dytuh}
If in an iteration $i$ the else-case starting in line 8 is executed, then $M$ decreases in iteration $i$.
\end{Claim}
\begin{proof}
If the if-condition in line 12 is not satisfied, $|N_H(s)|$ decreases in iteration $i$ and $\X$ remains unchanged, so suppose otherwise.  If the if-condition in line 14 is satisfied, then in line 15 $\X$ is replaced by a biset containing $\X$ and $v$ leaves $N_H(s)-X_O$, so $|N_H(s)-X_O|$ decreases and $N_H(s)$ remains unchanged. Otherwise, by Lemma \ref{new34}, we have $w(\X)=w(\Y)$. Therefore, if the if-condition in line 17 is satisfied, $\X$ is replaced by $\Y$ in line 18 and $\Y$ satisfies $X_O\cup\{v\}\subseteq Y_O$, so $|N_H(s)-X_O|$ decreases and $N_H(s)$ remains unchanged. So suppose that the else-case starting in line 19 is executed. If the if-condition in line 22 is not satisfied, $|N_H(s)|$ decreases in iteration $i$ and $\X$ remains unchanged, so suppose otherwise. Then in line 23 $\X$ is replaced by a biset containing $\X$ and $v$ leaves $N_H(s)-X_O$, so $|N_H(s)-X_O|$ decreases and $N_H(s)$ remains unchanged.
\qed
\end{proof}
\begin{Claim}\label{dcuiq}
If in an iteration $i$ the if-condition in line 3 is satisfied, then either $M$ decreases in iteration $i$ or $M$ remains unchanged in iteration $i$ and decreases in iteration $i+1$.
\end{Claim}
\begin{proof}
If the if-condition in line 6 is not satisfied, $|N_H(s)|$ decreases in iteration $i$ and $\X$ remains unchanged, so suppose otherwise. By Proposition \ref{ana}$(c)$, $|X_O\cap N_H(s)|\leq 2$ before iteration $i$. As $\X$ is replaced by a horrifying biset in line 7, we have $|X_O\cap N_H(s)|\geq 2$ after iteration $i$. As $N_H(s)$ remains unchanged, $M$ does not increase in iteration $i$. As $\X$ is horrifying after iteration $i$, it follows that in iteration $i+1$ the else-case starting in line 8 is executed, so $M$ decreases in iteration $i+1$ by Claim \ref{dytuh}. 
\qed\end{proof}
Claims \ref{dytuh} and \ref{dcuiq} show that $M$ never increases and decreases in at least one of two consecutive iterations. Further observe that $M$ is always an integer satisfying $0\leq M \leq 2n$. It follows that the while-loop runs at most $4n$ times. This finishes the proof. 
\qed\end{proof}

\section{Acknowledgement}

We thank the anonymous referees to kindly ask us to simplify our algorithm. This allowed us to find a new approach and to significantly simplify the algorithm and the proofs.

%\fi

\end{document}